\PassOptionsToPackage{svgnames}{xcolor}
\documentclass[a4paper, 11pt]{article}
\usepackage[T1]{fontenc}

\usepackage{fullpage}
\usepackage[english]{babel}

\usepackage{todonotes}

\usepackage{latexsym}
\usepackage{amsmath}
\usepackage{amsthm}
\usepackage{amssymb}
\usepackage[normalem]{ulem}

\usepackage[linesnumbered,noend]{algorithm2e}
\RestyleAlgo{ruled} 
\SetKwComment{Comment}{/* }{ */}
\SetKwInput{KwData}{Input}
\SetKwInput{KwResult}{Output}

\usepackage[hidelinks]{hyperref}
\usepackage{cleveref}
\usepackage{url}
\usepackage{comment}
\usepackage{csquotes}

\usepackage{tikz}
\usetikzlibrary{shapes,decorations}
\usetikzlibrary{graphs, graphs.standard}
\usetikzlibrary{positioning}
\usetikzlibrary{scopes}
\usetikzlibrary{fit,backgrounds}
\tikzstyle{none}=[]
\tikzstyle{base}=[circle, fill=black!12, draw, inner sep=0pt, minimum width=8pt, minimum height=8pt, line width=0.5pt]
\tikzstyle{wideBase}=[circle, fill=black!12, draw, inner sep=0pt, minimum width=12pt, minimum height=12pt, line width=0.5pt]
\tikzstyle{dashStyle}=[-, dashed]
\tikzstyle{dotStyle}=[-,dotted]
\tikzstyle{baseEdge}=[-, draw=black, line width=2pt]
\tikzstyle{wideLine}=[-, draw=black, line width=2pt]
\tikzstyle{arrow}=[->, draw=black, line width=2pt]
\tikzstyle{slimArrow}=[->, draw=black, line width=1pt]
\tikzstyle{nodeLabel}=[shape=rectangle, fill=white, minimum width=8pt, minimum height=8pt, inner sep=0pt]
\pgfdeclarelayer{edgelayer}
\pgfdeclarelayer{nodelayer}
\pgfsetlayers{edgelayer,nodelayer,main}

\definecolor{softgreen}{HTML}{2cab27}
\definecolor{sanguine}{HTML}{eb5a21}
\definecolor{softyellow}{HTML}{e6e337}
\definecolor{coolpink}{RGB}{228, 127, 226}
\definecolor{coolgreen}{RGB}{127, 211, 125}
\definecolor{coolbrown}{RGB}{204, 174, 161}
\definecolor{coolorange}{RGB}{248, 185, 126}
\definecolor{coolred}{RGB}{255, 127, 125}

\usepackage{caption}

\title{Directed hypergraph connectivity augmentation\\ by hyperarc reorientations}
\author{Moritz M\"uhlenthaler, Benjamin Peyrille, Zolt\'an Szigeti}
\date{\today}

\newtheorem{theorem}{Theorem}

\newtheorem{lemma}[theorem]{Lemma}

\newtheorem{claim}{Claim}[theorem]
\newtheorem{gclaim}{Claim}
\crefname{lemma}{Lemma}{Lemmas}
\crefname{theorem}{Theorem}{Theorems}
\crefname{corollary}{Corollary}{Corollaries}
\crefname{definition}{Definition}{Definitions}
\crefname{conjecture}{Conjecture}{Conjectures}
\crefname{figure}{Figure}{Figures}
\crefname{algorithm}{Algorithm}{Algorithms}
\crefname{remark}{Remark}{Remarks}
\crefname{claim}{Claim}{Claims}
\crefname{gclaim}{Claim}{Claims}
\crefname{section}{Section}{Sections}

\newenvironment{claimproof}
  {\begin{proof}}
  {\end{proof}}

\newcommand{\Tight}[1]{\ensuremath{\mathcal{T}_#1}}
\newcommand{\Dangerous}[1]{\ensuremath{\mathcal{D}_#1}}
\newcommand{\Minimal}[1][{}]{\ensuremath{\mathcal{M}_#1}}
\newcommand{\R}{\ensuremath{\mathcal{R}}}
\newcommand{\Hg}{\ensuremath{\mathcal{H}}}

\newcommand{\cev}[1]{\reflectbox{\ensuremath{\vec{\reflectbox{\ensuremath{#1}}}}}}

\begin{document}


\maketitle

\begin{abstract}
The orientation theorem of Nash-Williams states that an undirected graph admits a $k$-arc-connected orientation if and only if it is $2k$-edge-connected. Recently, 
Ito et al.~\cite{ito2022monotone} showed that any orientation of an undirected $2k$-edge-connected graph can be transformed into a $k$-arc-connected orientation by reorienting one arc at a time without decreasing the arc-connectivity at any step, thus providing an algorithmic proof of Nash-Williams' theorem. 
We generalize their result to hypergraphs and therefore provide an algorithmic proof of the characterization of hypergraphs with a $k$-hyperarc-connected orientation originally given by Frank et al.~\cite{frank2003orientation}. We prove that any orientation of an undirected $(k,k)$-partition-connected hypergraph can be transformed into a $k$-hyperarc-connected orientation by reorienting one hyperarc at a time without decreasing the hyperarc-connectivity in any step. Furthermore, we provide a simple combinatorial algorithm for computing such a transformation in polynomial time.
\end{abstract}

\section{Introduction}

An undirected (resp.,~directed) graph is \emph{$k$-edge-connected} (resp., $k$-arc-connected) if every ordered pair of vertices is connected by at least $k$ edge-disjoint (resp.,~arc-disjoint) paths.
For any undirected graph we may obtain an \emph{orientation} of that graph by orienting each of its edges in one of the two possible ways.
In a classic theorem, Nash-Williams~\cite{nash1960orientations} showed that an undirected graph admits a $k$-arc-connected orientation if and only if it is $2k$-edge-connected. Frank \cite{frank1982note} introduced reorientations of undirected graphs and proved that any two $k$-arc-connected orientations of a given undirected graph can be transformed into one another by reorienting directed cycles and directed paths. Recently, Ito et al.~\cite{ito2022monotone} showed that any orientation of an undirected $2k$-edge-connected graph can be transformed into a $k$-arc connected orientation by changing the orientation of a single arc in each step. Furthermore, the arc-connectivity of each orientation is at least as high as that of the previous orientations. Their result gives an algorithmic proof of the non-trivial direction of Nash-Williams' theorem. In fact, Ito et al.~showed such a transformation of length $kn^3$ can be computed in polynomial time, where $n$ is the number of vertices of the input graph. 

We propose a generalization of the result of Ito et al.~to hypergraphs.  
An undirected hypergraph is $(k,k)$-partition connected if for any partition $\mathcal{P} = \{V_1, \ldots, V_t\}$ of the vertex set, the number of hyperedges intersecting at least two members of $\mathcal{P}$ is at least $kt$  (see Section \ref{sec:preliminaries} for other definitions).
Frank et al.~\cite{frank2003orientation} generalized Nash-Williams' theorem by showing that
an undirected hypergraph admits a $k$-hyperarc-connected orientation if and only if it is $(k,k)$-partition connected.
We prove that any orientation of an undirected $(k,k)$-partition-connected hypergraph can be transformed into a $k$-hyperarc-connected orientation by reorienting one hyperarc at a time without decreasing the hyperarc-connectivity at any step.
Furthermore, such a transformation of length at most $kn^3$ can be computed in time polynomial in the size of the input hypergraph. This provides an algorithmic proof for the characterization of hypergraphs that admit a $k$-hyperarc-connected orientations by Frank et al.~\cite{frank2003orientation}. To our knowledge, our result is the first efficient algorithm for computing a $k$-hyperarc-connected orientation of a hypergraph if one exists.

Our approach is similar to that of Ito et al.~\cite{ito2022monotone}: Starting from any orientation of the $(k, k)$-partition-connected input hypergraph, we repeatedly reorient hyperarcs along carefully chosen hyperpaths until a $k$-hyperarc-connected orientation is obtained. 
Our approach generalizes the work of Ito et al.~\cite{ito2022monotone}, provides a simpler proof of their result for graphs and explicitly gives an algorithm.

\section{Definitions}
\label{sec:preliminaries}

Let $\Hg := (V, \mathcal{E})$ be a \emph{hypergraph}, where $\mathcal{E}$ is a multiset of subsets of $V$. The elements of $\mathcal{E}$ are called \emph{hyperedges}.
The \emph{degree} $d_{\Hg}(X)$ of a set $X$ of vertices of $\Hg$ is the number of hyperedges $Z$ such that $Z$ intersects both $X$ and $V - X$.
The hypergraph $\Hg$ is \emph{$k$-hyperedge-connected} if for any non-empty proper subset $X$ of vertices, we have $d_{\Hg}(X) \geq k$. The \emph{hyperedge-connectivity} of $\Hg$, denoted by $\lambda(\Hg)$, is the maximum integer $k$ such that $\Hg$ is $k$-hyperedge-connected.

We can orient a hyperedge $X \in \mathcal{E}$ towards $v \in X$ to obtain a \emph{hyperarc} $(X - v, v)$, where $X-v$ is the set of \emph{tail vertices} of the hyperarc and $v$ is the \emph{head vertex} of the hyperarc. By orienting each edge of $\Hg$ we obtain an \emph{orientation} $\vec{\Hg} = (V, \mathcal{A})$ of $\Hg$ which is a \emph{directed hypergraph}. Our definition of a directed hypergraph coincides with the B-hypergraph of Gallo et al.~in~\cite{gallo1993directed}.
Let $X \subsetneq V$ with $X \neq \emptyset$.
The \emph{in-degree} $d^-_{\vec{\Hg}}(X)$ of $X$ is the number of hyperarcs $(Y, v)$ of $\vec{\Hg}$ such that the head vertex $v$ is in $X$ and at least one tail vertex of $Y$ is not contained in $X$, that is, $Y - X \neq \emptyset$.
The \emph{out-degree} $d^+_{\vec{\Hg}}(X)$ of $X$ is the number of hyperarcs $(Y, v)$ of $\vec{\Hg}$ such that $v \notin X$ and $X \cap Y \neq \emptyset$.
We use $d^-_{\vec{\Hg}}(v)$ for $d^-_{\vec{\Hg}}(\{v\})$ and $d^+_{\vec{\Hg}}(v)$ for $d^+_{\vec{\Hg}}(\{v\}).$
Note that $d^-_{\vec{\Hg}}(X) = d^+_{\vec{\Hg}}(V - X)$.
We say that $\vec{\Hg}$ is \emph{$k$-hyperarc-connected} if for any non-empty proper subset $X$ of vertices, we have $d^+_{\vec{\Hg}}(X)\ge k.$
The \emph{hyperarc-connectivity} of $\vec{\Hg}$, denoted by $\lambda(\vec{\Hg})$, is the maximum integer $k$ such that $\vec{\Hg}$ is $k$-hyperarc-connected.
When it is clear from the context what hypergraph and what orientation we are considering, we will omit the mention of $\Hg$ and $\vec{\Hg}$ from the degree functions.

\begin{figure}[ht]
\centering
\begin{tikzpicture}[scale=1.0]
	\begin{pgfonlayer}{nodelayer}
		\node [style=wideBase] (0) at (-3, 0) {};
		\node [style=wideBase] (1) at (-3, 3) {};
		\node [style=wideBase] (2) at (0, 3) {};
		\node [style=wideBase] (3) at (0, 0) {};
		\node [style=wideBase] (4) at (3, 0) {};
		\node [style=wideBase] (5) at (3, 3) {};
		\node [style=none] (6) at (-2, 2) {};
		\node [style=none] (7) at (-1.25, 0.5) {};
		\node [style=none] (8) at (2, 0.75) {};
	\end{pgfonlayer}
	\begin{pgfonlayer}{edgelayer}
		\draw [style=baseEdge, draw={coolorange!125}, in=-45, out=30] (0) to (6.center);
		\draw [style=baseEdge, draw={coolorange!125}, in=-120, out=-45] (6.center) to (2);
		\draw [style=arrow, draw={coolpink!125}, in=135, out=45, looseness=0.75] (1) to (2);
		\draw [style=baseEdge, draw={softyellow!125}, in=30, out=75, looseness=1.50] (3) to (7.center);
		\draw [style=baseEdge, draw={softyellow!125}, in=30, out=-180] (5) to (7.center);
		\draw [style=arrow, draw={softyellow!125}, in=0, out=-150, looseness=0.75] (7.center) to (0);
		\draw [style=arrow, draw={coolred!125}, in=-60, out=60] (4) to (5);
		\draw [style=baseEdge, draw={coolbrown!125}, in=135, out=-60, looseness=0.75] (2) to (8.center);
		\draw [style=baseEdge, draw={coolbrown!125}, in=60, out=135, looseness=1.25] (8.center) to (3);
		\draw [style=arrow, draw={coolbrown!125}, in=135, out=-45, looseness=0.75] (8.center) to (4);
		\draw [style=baseEdge, draw={coolbrown!125}, in=-150, out=135, looseness=1.25] (8.center) to (5);
		\draw [style=arrow, draw={coolorange!125}] (6.center) to (1);
	\end{pgfonlayer}
\end{tikzpicture}
\caption{Directed hypergraph with 5 hyperarcs}
\end{figure}
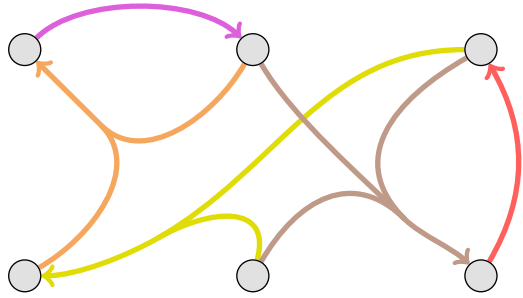

For $s,t\in V$, a  subset $X$ of $V$ that contains $s$ but not $t$ is called an \emph{$(s,t)$-separator}.
Let $\mathcal{P}$ be a partition of $V$. We denote by $e_{\Hg}(\mathcal{P})$ the number of hyperedges of $\Hg$ that intersect at least two members of $\mathcal{P}$.
The hypergraph $\Hg$ is \emph{$(k,k)$-partition-connected} if for any partition $\mathcal{P}$ of $V,$ we have $e_{\Hg}(\mathcal{P}) \geq k |\mathcal{P}|$. Applying $(k,k)$-partition-connectivity to any partition of size 2 shows that $(k,k)$-partition-connectivity implies $2k$-hyperedge-connectivity. The converse is not true: consider the hypergraph with vertex set  $\{a,b,c\}$ and two identical hyperedges $\{a,b,c\}$. This graph is $2$-hyperedge-connected, but not $(1,1)$-partition-connected. 

Let $P = (A_1, a_1), (A_2, a_2), \ldots, (A_{\ell}, a_{\ell})$ be a finite sequence of hyperarcs of $\vec{\Hg}$. Then $P$ is an \emph{$(s,t)$-hyperpath} if $s \in A_1$, $t = a_{\ell}$,   $a_i \in A_{i+1}$ for $i \in 1, \ldots, \ell-1$, and the vertices $a_i$ are all distinct. An example of an $(s, t)$-hyperpath is shown in Figure~\ref{fig:hyperpath-example}. 
If $s=t$ then the $(s,t)$-hyperpath is called a \emph{directed hypercycle}.
The \emph{trimming} of an $(s,t)$-hyperpath $P$ is a directed $(s, t)$-path $P'$ whose vertices are in the order $s, a_1, a_2, \ldots, a_{\ell-1}, t$.
We \emph{reorient} a hyperarc $(X, v)$ of $\vec{\Hg}$ towards $u \in X$, $u \neq v$, by replacing $(X,v)$ by the hyperarc $(X - u + v, u)$.

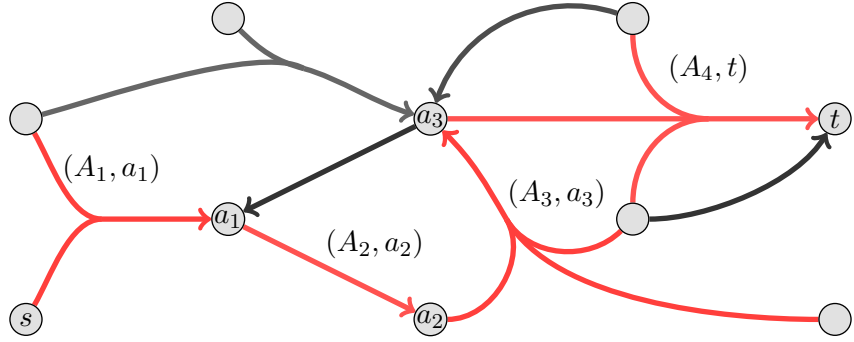
\begin{figure}
    \centering
\begin{tikzpicture}[scale=1.33]
	\begin{pgfonlayer}{nodelayer}
		\node [style=wideBase] (0) at (0, 0) {$s$};
		\node [style=wideBase] (1) at (2, 1) {$a_1$};
		\node [style=wideBase] (2) at (0, 2) {};
		\node [style=wideBase] (3) at (4, 2) {$a_3$};
		\node [style=wideBase] (4) at (2, 3) {};
		\node [style=wideBase] (5) at (4, 0) {$a_2$};
		\node [style=wideBase] (6) at (8, 2) {$t$};
		\node [style=wideBase] (7) at (6, 3) {};
		\node [style=wideBase] (8) at (6, 1) {};
		\node [style=wideBase] (9) at (8, 0) {};
		\node [style=none] (10) at (0.75, 1) {};
		\node [style=none] (11) at (0.85, 1.5) {$(A_1,a_1)$};
		\node [style=none] (12) at (4.75, 1) {};
		\node [style=none] (13) at (3.45, 0.75) {$(A_2,a_2)$};
		\node [style=none] (14) at (5.25, 1.25) {$(A_3,a_3)$};
		\node [style=none] (15) at (6.75, 2) {};
		\node [style=none] (16) at (6.75, 2.5) {$(A_4,t)$};
		\node [style=none] (17) at (2.75, 2.5) {};
	\end{pgfonlayer}
	\begin{pgfonlayer}{edgelayer}
		\draw [style=baseEdge, in=180, out=60, looseness=0.75, draw=coolred!150] (0) to (10.center);
		\draw [style=baseEdge, in=180, out=-60, looseness=0.75, draw=coolred!150] (2) to (10.center);
		\draw [style=arrow, in=180, out=0, draw=coolred!150] (10.center) to (1);
		\draw [style=arrow, , draw=coolred!135] (1) to (5);
		\draw [style=baseEdge, in=-60, out=-180, looseness=0.75, draw=coolred!150] (9) to (12.center);
		\draw [style=baseEdge, in=-60, out=0, draw=coolred!150] (5) to (12.center);
		\draw [style=baseEdge, in=-60, out=-135, draw=coolred!150] (8) to (12.center);
		\draw [style=arrow, in=-45, out=120, draw=coolred!150] (12.center) to (3);
		\draw [style=arrow, draw=coolred!135] (15.center) to (6);
		\draw [style=baseEdge, in=-180, out=-90, draw=coolred!135] (7) to (15.center);
		\draw [style=baseEdge, in=180, out=0, draw=coolred!135] (3) to (15.center);
		\draw [style=arrow, in=-120, out=0, looseness=0.75, draw=black!80] (8) to (6);
		\draw [style=baseEdge, in=150, out=15, looseness=0.50, draw=black!60] (2) to (17.center);
		\draw [style=baseEdge, in=165, out=-45, draw=black!60] (4) to (17.center);
		\draw [style=arrow, in=165, out=-15, looseness=0.75, draw=black!60] (17.center) to (3);
		\draw [style=arrow, draw=black!80] (3) to (1);
		\draw [style=arrow, bend right=45, draw=black!70] (7) to (3);
        \draw [style=baseEdge, in=180, out=90, draw=coolred!135] (8) to (15.center);
	\end{pgfonlayer}
\end{tikzpicture}
    \caption{$(s,t)$-Hyperpath in red}
    \label{fig:hyperpath-example}
\end{figure}

We say that a function $f$ is \emph{submodular} if for any two vertex sets $X,Y,$ we have $f(X) + f(Y) \geq f(X \cap Y) + f(X \cup Y)$.
Remark that each presented degree function is submodular (see \cite{frank2011connections}).
Two vertex sets $X, Y$ are said to be \emph{crossing} if $X \cap Y \neq \emptyset$, $X \cup Y \neq V$ and $X - Y \neq \emptyset \neq Y - X$. 

\section{Previous work}

The following fundamental result on orientations of undirected graphs is due to Nash-Williams~\cite{nash1960orientations}. The original proof of Nash-Williams is not algorithmic, but meanwhile many different proofs are known, including algorithmic ones~\cite{lovasz1979combinatorial}.

\begin{theorem}{{\cite{nash1960orientations}}}
\label{nash_williams_weak_orientation_theorem}
An undirected graph $G$ has a $k$-edge-connected orientation if and only if $G$ is $2k$-edge-connected.
\end{theorem}

Since then, orientations of undirected graphs have been studied widely \cite{bernath2008recent,szigeti2010orientations,frank1980orientation}.

\medskip

In his seminal work~\cite{frank1982note}, Frank considered reorientations of $k$-arc-connected directed graphs.

\begin{theorem}{{\cite[Theorem 1]{frank1982note}}}
\label{thm:frank_reorientation_graphs}
    If $\vec{G}$ and $\vec{G}'$ are two $k$-arc-connected orientations of an undirected graph $G$ then there is a sequence $\vec{G} = \vec{G}_0, \vec{G}_1,\ldots,\vec{G}_{\ell}=\vec{G}'$ of $k$-arc-connected orientations of $G$ such that each $\vec{G}_i$ arises from $\vec{G}_{i-1}$ by reversing one directed cycle or directed path.
\end{theorem}

In \cite{ito2022monotone}, Ito et al.~noticed that given two $k$-arc-connected orientations $\vec{G}$ and $\vec{G}'$ of an undirected graph $G$, it is not always possible to get  $\vec{G}'$ from  $\vec{G}$ by reorienting arcs one by one such that each intermediate orientation to be $k$-arc-connected. For an example, see the two $1$-arc-connected orientations of a cycle of length $3.$ However, for a $(2k+2)$-edge-connected undirected graph, they proved that such a reorientation is possible, using  \cref{thm:frank_reorientation_graphs} and the following result as a subroutine.
They mentioned that this result also provides a new algorithmic proof of \cref{nash_williams_weak_orientation_theorem}.

\clearpage

\begin{theorem}{{\cite[Theorem 1.1]{ito2022monotone}}}
\label{thm:conn_augm_total_graph}
    Let $G = (V, E)$ be a $2k$-edge-connected graph and $\vec{G}$ an orientation of $G$. 
    Then there exists a sequence $\vec{G}=\vec{G}_0, \vec{G}_1, \vec{G}_2, \ldots, \vec{G}_\ell$ of orientations of $G$ such that for $1 \leq i \leq \ell$, the orientation $\vec{G}_i$ is obtained from $\vec{G}_{i-1}$ by reorienting a single arc  of $\vec{G}_{i-1}$, $\lambda(\vec{G}_{i-1}) \leq \lambda(\vec{G}_i)$, 
    and $\lambda(\vec{G}_{\ell}) = k$. 
    Furthermore, such a sequence with $\ell \leq (k-\lambda(\vec{G}))|V|^3$ can be computed in time polynomial in the size of $G$.
\end{theorem}

Using \cref{thm:frank_reorientation_graphs,thm:conn_augm_total_graph}, the following result is given by Ito et al.~\cite{ito2022monotone}.

\begin{theorem}{{\cite[Theorem 1.2]{ito2022monotone}}}
\label{thm:reconfiguration_graph}
If $\vec{G}$ and $\vec{G}'$ are two $k$-arc-connected orientations of a $(2k+2)$-edge-connected graph $G$ then there exists a sequence $\vec{G} = \vec{G}_0, \vec{G}_1,\ldots,\vec{G}_{\ell}=\vec{G}'$ of $k$-arc-connected orientations of $G$ such that each $\vec{G}_i$ arises from $\vec{G}_{i-1}$ by reorienting a single arc.
\end{theorem}

In this paper, we study reorientations of hypergraphs.
The following characterization of hypergraphs that admit a $k$-hyperarc-connected orientation is due to Frank, Kir\'aly, and Kir\'aly~\cite{frank2003orientation}.

\begin{theorem}{{\cite[Corollary 3.5]{frank2003orientation}}}
\label{KPP_hypergraph}
A hypergraph $\Hg$ admits a $k$-hyperarc-connected orientation if and only if $\Hg$ is $(k, k)$-partition-connected.
\end{theorem}

Remark that not only \cref{KPP_hypergraph} is the hypergraphic analogue of Nash-Williams' theorem but also a generalization of it.

\section{New results}
\label{sec:main}

Our aim is to generalize \cref{thm:conn_augm_total_graph} to hypergraphs, providing a combinatorial algorithmic proof for \cref{KPP_hypergraph} as well as an explicit algorithm for connectivity-augmentation of directed hypergraphs by reorienting hyperarcs. The main result of the present paper is the following theorem.

\begin{theorem}
\label{thm:conn_augm_total_hypergraph}
Let $\Hg = (V,\mathcal{E})$ be a $(k,k)$-partition-connected undirected hypergraph and $\vec{\Hg}$ an orientation of $\Hg$. Then there exists a sequence $\vec{\Hg} = \vec{\Hg}_0, \vec{\Hg}_1, \ldots, \vec{\Hg}_{\ell}$ of orientations of $\Hg$ such that for $1 \leq i \leq \ell$, the orientation $\vec{\Hg}_{i}$ is obtained from $\vec{\Hg}_{i-1}$ by reorienting a single hyperarc of $\vec{\Hg}_{i-1}$, $\lambda(\vec{\Hg}_{i-1}) \leq \lambda(\vec{\Hg}_i)$, and $\lambda(\vec{\Hg}_{\ell}) = k$. Furthermore, such a sequence with $\ell \leq (k - \lambda(\vec{\Hg}))|V|^3$ can be computed in time polynomial in the size of $\Hg$.
\end{theorem}

\cref{thm:conn_augm_total_hypergraph} provides the proof of sufficiency in \cref{KPP_hypergraph}.
Notice that the necessity in \cref{KPP_hypergraph} is straight-forward. Indeed, let $\vec{\Hg}$ be a $k$-hyperarc-connected orientation of $\Hg$. Then, for any partition $\mathcal{P}$ of $V$, we have $e_\Hg(\mathcal{P}) = \sum_{X \in \mathcal{P}} d^-_{\vec{\Hg}}(X) \geq k |\mathcal{P}|$ and hence $\Hg$ is $(k,k)$-partition-connected.
\medskip

\cref{thm:conn_augm_total_hypergraph} is obtained by applying  $k - \lambda({\vec{\Hg}})$ times \cref{thm:conn_augm_hypergraph}.

\begin{theorem}
\label{thm:conn_augm_hypergraph}
    Let $\Hg = (V, \mathcal{E})$ be a $(k+1, k+1)$-partition-connected hypergraph and $\vec{\Hg}$ a $k$-hyperarc-connected orientation of $\Hg$. 
    Then there exists a sequence $\vec{\Hg} = \vec{\Hg}_0, \vec{\Hg}_1, \vec{\Hg}_2, \ldots, \vec{\Hg}_\ell$ of orientations of $\Hg$ such that for $1 \leq i \leq \ell$, the orientation $\vec{\Hg}_i$ is obtained from $\vec{\Hg}_{i-1}$ by reorienting a single hyperarc  of $\vec{\Hg}_{i-1}$, $\lambda(\vec{\Hg}_{i-1}) \leq \lambda(\vec{\Hg}_i)$, 
    and $\lambda(\vec{\Hg}_{\ell}) = k+1$. 
    Furthermore, such a sequence with $\ell \leq |V|^3$ can be computed in time polynomial in the size of $\Hg$.
\end{theorem}

The proof of \cref{thm:conn_augm_hypergraph} can be found in \cref{sec:preliminary,sec:main_algorithm,sec:existence,sec:finding,sec:admissability,sec:complexity}.

\medskip

We can also generalize  \cref{thm:frank_reorientation_graphs} to hypergraphs as follows.

\begin{theorem}
\label{thm:frank_reorientation_hypergraphs}
    If $\vec{\Hg}$ and $\vec{\Hg}'$ are two $k$-hyperarc-connected orientations of an undirected hypergraph $\Hg$ then there is a sequence $\vec{\Hg} = \vec{\Hg}_0, \vec{\Hg}_1,\ldots,\vec{\Hg}_{\ell}=\vec{\Hg}'$ of $k$-hyperarc-connected orientations of $\Hg$ such that each $\vec{\Hg}_i$ arises from $\vec{\Hg}_{i-1}$ by reversing one directed hypercycle or  hyperpath.
\end{theorem}

The proof of \cref{thm:frank_reorientation_hypergraphs} is similar to that of Frank in \cite{frank1980orientation} for \cref{thm:frank_reorientation_graphs}. We briefly explain what must be changed in that proof to apply it for hypergraphs. Frank first considers the case when $d^-_{\vec{G}}(v) = d^-_{\vec{G}'}(v)$ for all vertices of $G$. He calls \emph{red} the arcs of $\vec{G}$ that are oriented differently in $\vec{G}$ and in $\vec{G}'$. Notice that in the directed graph $D^r$ induced by the red arcs, the in-degree is equal to the out-degree for every vertex. Frank then decomposes $D^r$ into directed cycles and reorients each. For hypergraphs, we call \emph{red} the hyperedges of $\Hg$ that are oriented differently in $\vec{\Hg}$ and in $\vec{\Hg}'$. We define a directed graph $D^r$ on the vertices of $\Hg$. For each red hyperedge $X$, we add the arc $uv$ to $D^r$ where $u$ and $v$ are the head vertices of the corresponding hyperarc $(X-v, v)$ of $\vec{\Hg}$ and hyperarc $(X - u, u)$ of $\vec{\Hg}'$. We then proceed to decompose $D^r$ into directed cycles and reorient the corresponding directed hypercycles. The rest of the proof is identical.

\medskip

Finally, we generalize \cref{thm:reconfiguration_graph} to hypergraphs.

\begin{theorem}
\label{thm:reconfiguration_hypergraph}
If $\vec{\Hg}$ and $\vec{\Hg}'$ are two $k$-hyperarc-connected orientations of a $(k+1, k+1)$-partition-connected hypergraph $\Hg$ then there exists a sequence $\vec{\Hg} = \vec{\Hg}_0, \vec{\Hg}_1,\ldots,\vec{\Hg}_{\ell}=\vec{\Hg}'$ of $k$-hyperarc-connected orientations of $\Hg$ such that each $\vec{\Hg}_i$ arises from $\vec{\Hg}_{i-1}$ by reorienting a single hyperarc.
\end{theorem}

Similarly to how \cref{thm:reconfiguration_graph} is obtained in \cite{ito2022monotone} from \cref{thm:conn_augm_total_graph,thm:frank_reorientation_graphs}, \cref{thm:reconfiguration_hypergraph} is obtained by using \cref{thm:conn_augm_total_hypergraph,thm:frank_reorientation_hypergraphs}.

\section{Tight and dangerous sets}
\label{sec:preliminary}

For the remainder of this article we fix a $(k+1, k+1)$-partition-connected hypergraph $\Hg = (V, \mathcal{E})$. In order to avoid the verification of $X \cup Y = V$ for whether two sets $X, Y \subseteq V$ are crossing, we fix an arbitrary vertex $r \in V$ in the following definitions.
To prove \cref{thm:conn_augm_hypergraph}, we follow the paper of Ito et al.~\cite{ito2022monotone} and introduce the following sets with respect to an arbitrary $k$-hyperarc-connected orientation $\vec{\Hg}$ of $\Hg$, which will be clear from the context.

\begin{align*}
\Tight{-} &= \{X \subseteq V - r : d^-(X) = k\} \cup \{V\}\\
\Tight{+} &= \{X \subseteq V - r : d^+(X) = k\} \cup \{V\}\\
\Minimal{-} &= \text{ inclusion-wise minimal members of } \Tight{-}\\
\Minimal{+} &= \text{ inclusion-wise minimal members of } \Tight{+}\\
\Minimal{} &= \text{ inclusion-wise minimal members of }\Minimal{-} \cup \Minimal{+}\\
\Dangerous{-} &= \{X \subseteq V - r : d^-(X) = k+1\}\\
\Dangerous{+} &= \{X \subseteq V - r : d^+(X) = k+1\}\\
\R &= \{ R \subseteq V - r : \text{$R$ is inclusion wise minimal such that} \\
 &\qquad \hskip 1.9truecm \text{either $R \in \Tight{-}$ and $R$ contains a member of $\Tight{+}$} \\
 &\qquad \hskip 1.9truecm \text{or $R \in \Tight{+}$ and $R$ contains a member of $\Tight{-}$} \}
\end{align*}

We present briefly some properties of these sets that follow from submodularity of the degree functions.

\begin{gclaim}
\label{clm:inter_union_tight_is_tight}\label{clm:substraction_tight_is_tight}

    Let $X, Y$ be two crossing  sets in $V$. 
    
    (a) If $X, Y\in \Tight{-}$ then $X \cup Y$ and $X \cap Y\in\Tight{-}$. 
 
    (b) If $X, Y\in \Tight{+}$ then $X \cup Y$ and $X \cap Y\in\Tight{+}$. 
    
    (c) If $X\in\Tight{-}$ and  $Y\in\Tight{+}$ then $X - Y \in \Tight{-}$ and $Y - X \in \Tight{+}$.
  
    (d) If $X\in\Minimal{-}$ and  $Y\in\Tight{+}\cup\Dangerous{+}$ then $Y\in\Dangerous{+}$ and $Y - X \in \Tight{+}$.

    (e) If $X\in\Tight{-}\cup\Dangerous{-}$ and $Y\in\Minimal{+}$  then $X\in\Dangerous{-}$ and $X - Y \in \Tight{-}$.
\end{gclaim}

\begin{proof}
(a) By $X, Y \in \Tight{-}$, submodularity of $d^-$, $X\cap Y\neq\emptyset,$ $X\cup Y\neq V,$ and $\lambda(\vec{\mathcal{H}})=k,$ we have
\begin{equation*}
    k + k = d^-(X) + d^-(Y) \geq d^-(X \cup Y) + d^-(X \cap Y) \geq k + k.
\end{equation*}
Hence equality holds everywhere, so we have $X \cup Y, X \cap Y \in \Tight{-}$.

(b) can be proved in the same way as (a) by  replacing $d^-$ by $d^+$.

(c) By $X\in \Tight{-}$, $Y\in \Tight{+}$, $d^-(V - Y) = d^+(Y)$, submodularity of $d^-$, $X-Y\neq\emptyset\neq Y-X,$ and $\lambda(\vec{\mathcal{H}})=k,$ we have 
    \begin{align*}
        k + k &= d^-(X) + d^-(V - Y)\\ 
        		&\geq d^-(X \cap (V - Y)) + d^-(X \cup (V - Y))\\ 
		&= d^-(X - Y) + d^+(Y - X)\\ 
		&\geq k + k.
    \end{align*}
Hence equality holds everywhere, so we have  $X - Y \in \Tight{-}$ and $Y - X \in \Tight{+}$.

(d) By $X\in\Minimal{-}$,  $Y\in\Tight{+}\cup\Dangerous{+}$, the submodularity of $d^-$, $\lambda(\vec{\mathcal{H}})=k,$ $X-Y\neq\emptyset\neq Y-X,$ and minimality of $X$, we obtain
\begin{align*}
    k + (k+1) & \geq d^-(X) + d^+(Y) \\
            & = d^-(X) + d^-(V - Y) \\
            & \geq d^-(X - Y) + d^-(V-(Y-X)) \\
            & = d^-(X-Y) + d^+(Y-X) \\
            & \geq (k+1) + k.
\end{align*}
It follows that $d^+(Y) = k+1$ and that $d^+(Y - X) = k$.

(e) can be proved similarly as (d).
\end{proof}

By \cref{clm:inter_union_tight_is_tight,clm:substraction_tight_is_tight}, each of the families $\Minimal{-}, \Minimal{+} \text{ and } \Minimal{}$ is a subpartition of $V$.

\section{Main algorithm}
\label{sec:main_algorithm}

We follow the framework presented by \cite{ito2022monotone} and adapt it for hypergraphs.

\medskip

{
\setlength\parindent{7.25pt}

\noindent
For $S \in \Minimal{-}$, a vertex $s \in S$ is a \emph{safe source} in $S$ if

\smallskip

(a) for every $s \in X \in \Tight{+}$, we have $S \subsetneq X$,

\smallskip

(b) for every $s \in X \in \Dangerous{+}$ such that $S - X \neq \emptyset$, there exists $Y \in \Tight{+}$ such that $s \not\in Y \subsetneq X$.

\medskip

\noindent
For $T \in \Minimal{+}$, a vertex $t \in T$ is a \emph{safe sink} in $T$ if

\smallskip

(a) for every $t \in X \in \Tight{-}$, we have $T \subsetneq X$,

\smallskip

(b) for every $t \in X \in \Dangerous{-}$ such that $T - X \neq \emptyset$, there exists $Y \in \Tight{-}$ such that $t \not\in Y \subsetneq X$.
}
\medskip

\noindent
Let $P$ be a directed $(s,t)$-hyperpath in $\vec{\Hg}$ and $R$ a set in $\R$. We say that $P$ is \emph{admissible} in $R$ if
\begin{itemize}
\setlength\itemsep{1pt}
	\item [(a)] $s$ is a safe source in a  set $S \in \Minimal{-}$, 
		$t$ is a safe sink in a  set $T \in \Minimal{+}$, 
		$S$ and $T$ are contained in $R$, 
	\item [(b)] reorienting all hyperarcs in $P$ one by one, 
 			from start to end if $R\in \R\cap\Tight{+}$ and 
 			from end to start if $R\in \R\cap\Tight{-}$, 
		never decreases the hyperarc-connectivity of the current directed hypergraph, 
	\item [(c)] after all hyperarcs of $P$ have been reoriented, the new orientation's 
		$\Minimal{}'$ set has 
			either less elements than $\Minimal{}$, 
			or the same number of elements as $\Minimal{}$ 
		but covers more vertices than $\Minimal{}$.
\end{itemize}
\medskip

\cref{alg:conn-augmentation-global} aims to compute a $(k+1)$-hyperarc-connected orientation of $\Hg$ from a given $k$-hyperarc-connected orientation $\vec{\Hg}$ of $\Hg$. The algorithm starts by selecting an arbitrary vertex $r$ of $\Hg$, which will allow us to compute the sets $\R, \Minimal{-}, \Minimal{+}$ with respect to $\vec{\Hg}$ and $r$. If $\Minimal{-} = \Minimal{+} = \{V\}$ then for all $X \subseteq V - r$, we have $d^-_{\vec{\Hg}}(X) \geq k+1$ and $d^+_{\vec{\Hg}}(X) \geq k+1$, this implies that we have a $(k+1)$-hyperarc-connected orientation and we are done.
Otherwise, we pick an arbitrary set $R$ in $\R$ and find an admissible $(s,t)$-hyperpath $P$ in $R$. Depending on whether $R$ is in $\Tight{+}$ or $\Tight{-}$, we reorient one by one each hyperarc of $P$ either from start to end or from end to start such that we do not decrease the hyperarc-connectivity of $\vec{\Hg}$ in each reorientation step. We repeat the process with the new orientation until $|\Minimal{}| = \{V\}$.
Since $\Minimal{}$ is a subpartition of $V$ and by the definition of an admissible hyperpath, the case $|\Minimal{}'| = |\Minimal{}|$ can only occur at most $|V|$ times until $|\Minimal{}'| < |\Minimal{}|$, which itself can only happen at most $|V|$ times. Thus, since no two hyperarcs in a hyperpath have the same head vertex, we reorient at most $|V|^3$ hyperarcs and hence the algorithm terminates.
It then suffices to show the existence of an admissible path and that reorienting its hyperarcs one by one does not decrease the hyperarc-connectivity to prove \cref{thm:conn_augm_hypergraph}.

\medskip

\begin{algorithm}[H]
\caption{Hypergraph connectivity augmentation algorithm}\label{alg:conn-augmentation-global}
\KwData{A $k$-hyperarc-connected orientation $\vec{\Hg}$ of a $(k+1,k+1)$-partition-connected hypergraph $\Hg$.}
\KwResult{A $(k+1)$-hyperarc-connected orientation $\vec{\Hg'}$ of $\Hg$.}
\SetKw{KwGoTo}{go to line}
Take an arbitrary vertex $r$ of $\Hg$\;
Let $\vec{\Hg'} := \vec{\Hg}$\;
Compute the sets $\R$, $\Minimal{-}$, $\Minimal{+}$ of $\vec{\Hg'}$\; \label{alglabel:compute_sets}
\If{$\Minimal{-} = \Minimal{+} = \{V\}$}{
    Return $\vec{\Hg'}$\;
}
Take a set $R \in \R$\;
Find an admissible $(s,t)$-hyperpath $P = (A_1, a_1),$ $(A_2, a_2),$ $\ldots, (A_\ell, a_\ell)$ in $R,$
using \cref{alg:adm-path-intight} if $R\in\Tight{-}$ and \cref{alg:adm-path-outtight} if $R\in\Tight{+}$;

\eIf{$R \in \Tight{-}$}{
    \For{$i = \ell, \ell-1, \ldots, 2$}{
        Reorient $(A_i, a_i)$ towards $a_{i-1}$\;
    }
    Reorient $(A_1, a_1)$ towards $s$\;
}{
    Reorient $(A_1, a_1)$ towards $s$\;
    \For{$i = 2, 3, \ldots, \ell$}{
        Reorient $(A_i, a_i)$ towards $a_{i-1}$;
    }
}
\KwGoTo \ref{alglabel:compute_sets}
\end{algorithm}

\section{Existence of safe sources and safe sinks}
\label{sec:existence}

We first show that any $S \in \Minimal{-}$ contains a safe source and any $T \in \Minimal{+}$ contains a safe sink. Notice that the proofs are not the same and, because of asymmetry of directed hypergraphs, the proof for the existence of safe sources cannot be used for the existence of safe sinks and vice versa.
We start with the existence of a safe source whose  proof  is inspired by a proof of Lov{\'a}sz  \cite{lovasz1973connectivity}.

\begin{lemma}
\label{lem:safe_source_exists}
For $S \in \Minimal{-}$, there exists a safe source $s$ in $S$.
\end{lemma}

\begin{proof}
Let $S \in \Minimal{-}$. If $S = V$ then $s = r$ is a safe source, so assume that $S \neq V$.
Let $\mathcal{X} = \{X_1, X_2, \cdots, X_\ell\}$ be a family of vertex sets such that $X_i \subseteq S, X_i \in \Tight{+} \cup \Dangerous{+}$ and $\mathcal{X}$ covers as many vertices of $S$ as possible and under this condition minimizes $\ell$.

\begin{claim}
\label{claim1_safe_source}
The family $\mathcal{X}$ does not cover all vertices of $S$.
\end{claim}
\begin{claimproof}
Suppose for a contradiction that $\bigcup_{X \in \mathcal{X}} X = S$.
Since $S \in \Tight{-}$, $S \neq V$, and $\Hg$ is $(k+1, k+1)$-partition-connected, we have $d^+(S) \ge 2(k+1)-d^-(S)= k+2$. Thus, $S \not\in \Tight{+} \cup \Dangerous{+}$ and $\ell \geq 2$. We also have
\begin{equation}\label{eq:barS}
    2 \leq d^+(S) - d^-(S) = d^-(\bar{S}) - d^+(\bar{S}).
\end{equation}
For every $1 \leq i \leq \ell$, let $Y_i = X_i - \bigcup_{i \neq j} X_j$. By the minimality of $\ell$, we have $Y_i \neq \emptyset$. Let $X_0 = Y_0 = \bar{S} \neq \emptyset$. We show that
\begin{equation}\label{eq:lovasz_counting_2}
    \sum_{i = 0}^\ell d^-(Y_i) \leq \sum_{i = 0}^\ell d^+(X_i).
\end{equation}
Indeed, since $\{Y_1, \ldots, Y_\ell\}$ is a subpartition of $V$, every hyperarc contributes at most 1 to the left-hand side. Moreover, if a hyperarc $(Z, v)$ contributes 1 to the left-hand side then $v \in Y_i$ for some $i \in \{0, 1, \ldots, \ell\}$ and there is some tail $u \in Z$ such that $u \in \bar{Y_i} = \bigcup_{j \neq i} X_j$. Therefore, $u \in X_j$ for some $j \neq i$, so $(Z,v)$ contributes at least 1 to the right-hand side and \eqref{eq:lovasz_counting_2} follows.

Furthermore, for $1 \leq i \leq \ell$, by $\ell \geq 2$, $X_i \in \Tight{+} \cup \Dangerous{+}$, $\lambda(\vec{\mathcal{H}})=k,$ and $\emptyset \neq Y_i \subsetneq S \in \Minimal{-}$, we have
\begin{equation}\label{eq:dXi}
    d^+(X_i) \leq k + 1 \leq d^-(Y_i).
\end{equation}

By equations \eqref{eq:barS}--\eqref{eq:dXi}, we obtain 
\begin{equation*}
    2 \leq d^-(\bar{S}) - d^+(\bar{S}) = d^-(Y_0) - d^+(X_0) \leq \sum_{i = 1}^\ell (d^+(X_i) - d^-(Y_i)) \leq 0,
\end{equation*}
a contradiction that proves \cref{claim1_safe_source}.
\end{claimproof}

By \cref{claim1_safe_source}, there exists a vertex $s \in S - \bigcup_{i=1}^\ell X_i$. We aim to show $s$ is a safe source in $S$. Suppose that $s \in X \in \Tight{+} \cup \Dangerous{+}$. By the definition of $s$ and $X$, we have $X - S \neq \emptyset$. Suppose now that $S - X \neq \emptyset$.
By \cref{clm:inter_union_tight_is_tight}(d), we have  $d^+(X) = k+1$ and $d^+(X - S) = k$. Then, for every $s \in X \in \Tight{+}$, we have $S \subsetneq X$, and for every $s \in X \in \Dangerous{+}$ such that $S - X \neq \emptyset$, we have $X - S \in \Tight{+}$. Therefore, $s$ is a safe source in $S$.
\end{proof}

We next show the existence of a safe sink whose proof is similar to the one found in \cite{ito2022monotone}.

\begin{lemma}
\label{lem:safe_sink_exists}
For $T \in \Minimal{+}$, there exists a safe sink $t$ in $T$.
\end{lemma}
\begin{proof}
Let $T \in \Minimal{+}$. If $T = V$ then $t = r$ is a safe sink, so we consider $T \neq V$. Let $\mathcal{Y} = \{Y_1, \ldots, Y_\alpha\}$ be the inclusionwise maximal vertex sets that are in $\Tight{-}$ and contained in $T$. Using \cref{clm:inter_union_tight_is_tight}, if two members of $\mathcal{Y}$ did cross then their union would be in $\Tight{-}$, thus the family $\mathcal{Y}$ is a subpartition.

Let $\mathcal{Z} = \{Z_1, \ldots, Z_\beta\}$ be the inclusionwise maximal vertex sets that are in $\Dangerous{-}$ and contained in $T - \bigcup_{Y \in \mathcal{Y}} Y$.
We show that $\mathcal{Z}$ is a subpartition; suppose for a contradiction that this is not the case. Then there exist distinct indices $i \neq j$ such that $Z_i \cap Z_j \neq \emptyset$. If $d^-(Z_i \cap Z_j) = k$ then $Z_i \cap Z_j \subseteq Y_{\ell} \in \mathcal{Y}$ for some $\ell$, contradicting that $Z_i \subseteq T - \bigcup_{Y \in \mathcal{Y}} Y$. Similarly, if $d^-(Z_i \cup Z_j) = k$ then $Z_i \cup Z_j \subseteq Y_{\ell} \in \mathcal{Y}$ for some $\ell$, contradicting that $Z_i \subseteq T - \bigcup_{Y \in \mathcal{Y}} Y$. Hence, by the submodularity of $d^-$ and $\lambda(\vec{\mathcal{H}})=k,$
\[
    2(k+1) = d^-(Z_i) + d^-(Z_j) \geq d^-(Z_i \cap Z_j) + d^-(Z_i \cup Z_j) \geq 2(k+1),
\]
so $d^-(Z_i \cup Z_j) = k+1$, which contradicts the maximality of $Z_i$.

\begin{claim}
\label{claim1_safe_sink}
    The family $\mathcal{Y} \cup \mathcal{Z}$ does not cover all vertices of $T$.
\end{claim}
\begin{claimproof}
Suppose this is not the case. Then $\mathcal{P} = \mathcal{Y} \cup \mathcal{Z} \cup \{\bar{T}\}$ is a partition of $V$. Notice that $d^-(\bar{T}) = d^+(T) = k$ and $e_\Hg(\mathcal{P}) = \sum_{X \in \mathcal{P}} d^-(X)$. Then, by $(k+1,k+1)$-partition-connectivity of $\mathcal{H}$, we have a contradiction:
    $\alpha k + \beta (k+1) + k = e_\Hg(\mathcal{P}) \geq |\mathcal{P}|(k+1) = (\alpha + \beta + 1)(k+1).$
\end{claimproof}

By \cref{claim1_safe_sink}, there exists a vertex $t \in T - \bigcup_{Y \in \mathcal{Y}} Y - \bigcup_{Z \in \mathcal{Z}} Z$. We first show that any set $X \in \Tight{-}$ that contains $t$ also contains $T$. If $X \subseteq T$ then $X \subseteq Y_i$ for some $i$, contradicting the choice of $t$. We hence suppose that $X - T \neq \emptyset$. If $T - X \neq \emptyset$ then by 
\cref{clm:inter_union_tight_is_tight}(c), we have  $T-X \in \Tight{+}$ which contradicts the minimality of $T$. This proves $T \subsetneq X$.

We now show that any set $X \in \Dangerous{-}$ such that $t \in X$ and $T - X \neq \emptyset$ contains a set $Y \subseteq X - t$ such that $Y \in \Tight{-}$. First, suppose $X \subsetneq T$. Since $t \in X \subsetneq T$, $X \in \Dangerous{-}$, and $t \not\in \bigcup_{Z \in \mathcal{Z}}Z$, we get that  $X$ is not contained in $T - \bigcup_{Y \in \mathcal{Y}} Y$. Thus, $X$ must intersect a set $Y_i$ for some $i$. By the submodularity of $d^-$, $\lambda(\vec{\Hg}) = k$, and the maximality of $Y_i$, we have 
\begin{equation*}
(k+1) + k = d^-(X) + d^-(Y_i) \geq d^-(X \cap Y_i) + d^-(X \cup Y_i) \geq k + (k+1).
\end{equation*}
 Therefore, $X \cap Y_i \in \Tight{-}.$
 Finally, suppose that $X - T \neq \emptyset$. Then, by the minimality of $T$ and \cref{clm:inter_union_tight_is_tight}(e), we have $X - T \in \Tight{-}$, proving $t$ is a safe sink in $T$.
\end{proof}

For $v \in V$, $Q^v_-$ denotes the inclusion-wise minimal set of $\Tight{-}$ containing $v$ and $Q^v_+$ denotes the inclusion-wise minimal set of $\Tight{+}$ containing $v$. Notice that there is no ambiguity since, by \cref{clm:inter_union_tight_is_tight}(a), the sets $Q^v_-$ et $Q^v_+$ are unique.

\begin{lemma}[see {\cite[Lemma 4.12]{ito2022monotone}}]
\label{lem:path_in_qsp_qtm}
\mbox{}
\begin{description}
    \item[(a)] For every $s \in V$ and $t \in Q^s_+$, there exists an $(s,t)$-hyperpath $P$ that does not leave $Q^s_+$.
    \item[(b)] For every $t \in V$ and $s \in Q^t_-$, there exists an $(s,t)$-hyperpath $P$  that does not leave $Q^t_-$.
\end{description}
\end{lemma}

\begin{proof}
(a) 
Suppose that the claim does not hold, so there exist $s \in V$ and $t \in Q^s_+$ such that any $(s,t)$-hyperpath leaves $Q^s_+$. Then there exists a vertex set $s \in Z \subseteq Q^s_+-t$ such that any hyperarc that leaves $Z$ also leaves $Q^s_+$. 
Therefore, we have $k \ge d^+(Q^s_+) \geq d^+(Z) \geq k$, which implies that $Z \in \Tight{+}$, contradicting the minimality of $Q^s_+$.

(b) can be proved similarly as (a).
Suppose that the claim does not hold, so there exist $t \in V-r$ and $s \in Q^t_-$ such that any $(s,t)$-hyperpath leaves $Q^t_-$. Then there exists a vertex set $t \in Z \subseteq Q^t_--s$ such that any hyperarc that enters $Z$ also enters $Q^t_-$. Therefore, we have  $k \ge d^-(Q^t_-) \geq d^-(Z) \geq k$, which shows that $Z\in \Tight{-}$, contradicting the minimality of $Q^t_-$.
\end{proof}

\begin{lemma}
\label{lem:safe_to_safe_kp1}
Let $R \in \R, S \in \Minimal{-}, T \in \Minimal{+}$ such that $S, T \subseteq R$. Furthermore, let $s$ be a safe source in $S$, and $t $  a safe sink in $T$. Then for all sets $X \subseteq V - r$ such that $s \in X, t \not\in X$, we have $d^+(X) \geq k + 1$, and for all sets $X \subseteq V - r$ such that $t \in X, s \not\in X$, we have $d^-(X) \geq k + 1$. 
\end{lemma}

\begin{proof}
If the claim does not hold then, by $\lambda(\vec{\mathcal{H}})=k,$ there exists $X \subseteq V - r$ such that either (a) $s \in X, t \not\in X$ and $X \in \Tight{+}$ or (b) $t \in X, s \not\in X$ and $X \in \Tight{-}$.

\begin{itemize}
    \item[(a)] By $S \in \Tight{-}$, $s \in X \in \Tight{+}$ and since $s$ is a safe source, we have $S \subsetneq X$. 
As $t \in R - X$ and $R \in \R$, we have that $X \nsubseteq R$ which implies $X - R \neq \emptyset$.
\begin{itemize}
    \item[-] If $R \in \Tight{-}$ then by $t \in R - X$, $X \in \Tight{+}$ and \cref{clm:substraction_tight_is_tight} on $X$ and $R$, we get that $R - X \in \Tight{-}$. If $X \cap T \neq \emptyset$, then, by \cref{clm:inter_union_tight_is_tight}, we get that $X \cap T \in \Tight{+}$, contradicting the minimality of $T$. Hence $T \subseteq R - X$. This, by $T \in \Tight{+}$ and $R - X \in \Tight{-}$, contradicts that $R \in \R$.
    \item[-] If $R \in \Tight{+}$, by $X \in \Tight{+}$ and \cref{clm:inter_union_tight_is_tight}, we get that $X \cap R \in \Tight{+}$. This, by $S \in \Tight{-}$, $S \subseteq R \cap X$ and $t \in R - X$, contradicts that $R \in \R$.
\end{itemize}

    \item[(b)] By $T \in \Tight{+}, t \in X \in \Tight{-}$ and since $t$ is a safe sink, we have $T \subsetneq X$. As $s \in R - X$ and $R \in \R$, we have that $X \not\subseteq R$ which implies $X - R \neq \emptyset$.
\begin{itemize}
    \item[-] If $R \in \Tight{-}$ then by $X \in \Tight{-}$ and \cref{clm:inter_union_tight_is_tight}, we get that $X \cap R \in \Tight{-}$. This, by $T \in \Tight{+}$, $T \subseteq X \cap R$ and $t \in R - X$, contradicts that $R \in \R$.
    \item[-] If $R \in \Tight{+}$, then by $s \in R - X$, $X \in \Tight{-}$ and \cref{clm:substraction_tight_is_tight} on $X$ and $R$, we get that $R - X \in \Tight{+}$. If $X \cap S \neq \emptyset$, then, by \cref{clm:inter_union_tight_is_tight}, we get that $X \cap S \in \Tight{-}$, contradicting the minimality of $S$. Hence $S \subseteq R - X$. This, by $S \in \Tight{-}$ and $R - X \in \Tight{+}$, contradicts that $R \in \R$.
\end{itemize}
\end{itemize}
\end{proof}

\section{Finding an admissible hyperpath}
\label{sec:finding}

In this section we present two algorithms to find an admissible hyperpath in a set $R$ of $\R$. We use \cref{alg:adm-path-intight} when $R$ is in $\Tight{-}$ and \cref{alg:adm-path-outtight} when $R$ is in $\Tight{+}$.

\medskip

\cref{alg:adm-path-intight} finds an $(s,t)$-hyperpath $P$ such that its trimming $P'$ neither leaves $R$ nor any set of $\Tight{+}$ it enters.
The algorithm works as follows. First, we choose a minimal set $S$ contained in $R$ (possibly $R$ itself), such that $S \in \Minimal{-}$ as well as a safe source $s$ in $S$.
We keep track of a set $Z$ of vertices we already explored (that contains $s$ initially), a \emph{search $s$-out-arborescence} $F$ and a set $V'$ of vertices that will potentially be explored and which is initialized to $R$.
While there is a hyperarc $(X, v)$ leaving $Z$ and whose head $v$ is in $V'$, we pick an arbitrary vertex $u \in X \cap Z$. This vertex $u$ will serve as the tail of the trimmed hyperarc. Note that $u$ exists by the choice of $(X,v)$. We mark the head vertex $v$ as explored, trim the hyperarc $(X, v)$ into the arc $uv$, and add it to $F$.
We ensure that if our search enters a set in $\Tight{+}$, we do not leave that set (see Line 11 of \cref{alg:adm-path-intight}).
We let $T$ be the last set in $\Tight{+}$ entered. We will see that $T \in \Minimal{+}$ and we find a safe sink $t$ in $T$.
Finally, we let $P'$ be the unique directed $(s, t)$-path in $F$ denoted by $F[s, t]$. Notice that $P'$ is the trimming of an $(s, t)$-hyperpath $P$ and the vertex set of $P'$ is contained in $R$. The algorithm returns $S, T, s, t, P$ and terminates.

\begin{algorithm}[H]
\caption{Admissible hyperpath in $R \in \Tight{-}$}\label{alg:adm-path-intight}
\KwData{A set $R \in \R \cap \Tight{-}$.}
\KwResult{A safe source $s$ in $S\in \Minimal{-}$ with $S\subseteq R$, a safe sink $t$ in $T\in\Minimal{+}$ with $T\subseteq R$, and an $(s,t)$-hyperpath $P$ that does not leave $R$.}
Take a set $S \in \Minimal{-}$ with $S \subseteq R$\;
Take a safe source $s$ in $S$\;
$Z := \{s\}$\;
$F := (Z, \emptyset)$\;
$V' := R$\;
\While{an $(X, v)$ hyperarc exists in $\vec{\Hg}$ such that $v\in V'-Z$ and $X \cap Z\neq\emptyset$}
	{
    Take an arbitrary vertex $u \in X \cap Z$\;
    $Z := Z \cup \{v\}$\;
    $F := F + uv$\;
    \If{$Q^v_+ \subsetneq V'$}
    	{
        $V' := Q^v_+$\;
    	}
	}
$T := V'$\;
Take a safe sink $t$ in $T$\;
$P' := F[s, t]$\;
Let $P$ be the hyperpath in $\vec{\Hg}$ corresponding to $P'$\;
Return $S,T,s,t,P$\;
\end{algorithm}

\medskip

We apply \cref{alg:adm-path-intight} to prove the following lemma that we need in order to show the existence of an admissible path in case of $R \in \R \cap \Tight{-}$. 

\begin{lemma}
\label{lem:algo_intight}
    If $R \in \R \cap \Tight{-}$ then there exist $S \in \Minimal{-}$, $T \in \Minimal{+}$,  a safe source $s$ in $S$, a safe sink $t$ in $T,$ and an $(s,t)$-hyperpath $P$ such that  $S \subseteq R$, $T \subsetneq R$,  $P$ does not leave $R$,  $Q^t_- = R$ and $\delta^+_{P'}(X) = \emptyset$ for every $X \in \Tight{+}$ with $s,t \not\in X$  where $P'$ is the unique directed $(s,t)$-path obtained from $P$ by trimming.
\end{lemma}

\begin{proof}
The algorithm starts by taking a safe source in $S$. Such a safe source exists by \cref{lem:safe_source_exists}. 

We show that \cref{alg:adm-path-intight} provides the required hyperpath $P$.
Let $T' \in \Minimal{+}$ with $T' \subsetneq R$ and $t'$ a safe sink in $T'$ which exists by \cref{lem:safe_sink_exists}. Then, by the definition of a safe sink and \cref{clm:inter_union_tight_is_tight}(a), $T' \subsetneq Q^{t'}_-\subseteq R.$ Thus, by minimality of $R$, we have $Q^{t'}_-=R.$ Hence, by \cref{lem:path_in_qsp_qtm}(b), $t'$ is reachable from $s$ by a hyperpath $P$  that does not leave $Q^{t'}_-$. This implies that the algorithm will eventually enter a set $T \in \Minimal{+}$, containing a safe sink $t$ by \cref{lem:safe_sink_exists}. Thus, the algorithm terminates with a safe source $s$ in $S\in \mathcal{M}_-$ with $S\subseteq R$, a safe sink $t$ in $T\in\mathcal{M}_+$ with $T\subseteq R$, and an $(s,t)$-hyperpath $P$ that does not leave $R$.

Recall that $P'=F[s,t].$
It remains to prove that for all $X \in \Tight{+}$ such that $s, t \not\in X$, $\delta^+_{P'}(X) = \emptyset$. Suppose that $\delta^+_{P'}(X) \neq \emptyset$. Since $s \not\in X$, the path $P'$ enters $X$. Let $uv$ be the first arc of $P'$ that enters $X$. Let $Y = Q^v_+$ and note that $P'$ enters $Y$. As $v \in X \cap Y, X, Y \in \Tight{+}$, we get, using \cref{clm:inter_union_tight_is_tight}(b) on $X$ and $Y$ and the minimality of $Y$, that $Y \subseteq X \subsetneq V - t$. As $Y \in \Tight{+}$ and $t \in T \in \Minimal{+}$, we get, by \cref{clm:inter_union_tight_is_tight}(b) and the minimality of $T$, that $Y$ and $T$ are disjoint. If $Y - R \neq \emptyset$ then, by $Y \in \Tight{+}$, $R \in \Tight{-}$, $t \in R - Y$ and by \cref{clm:substraction_tight_is_tight}(c), we obtain that $R - Y \in \Tight{-}$. Since $T \subseteq R - Y$, this contradicts $R \in \R$. Hence,  $Q^v_+=Y \subsetneq R$. Then, since $Q_v^+\subseteq Q_u^+$,  by the \emph{if} condition on line 10 of  \cref{alg:adm-path-intight}, we have that $F$ and hence $P'$ contains no arc leaving $Q^v_+$ and thus leaving $X$.  This contradicts $d^+_{P'}(X) \neq \emptyset$, completing the proof.
\end{proof}

\cref{alg:adm-path-outtight} finds an $(s,t)$-hyperpath $P$ such that its trimming $P'$ neither leaves $R$ nor enters any set of $\Tight{-}$ that does not contain neither $s$ nor $t$.

\begin{algorithm}[ht]
\caption{Admissible hyperpath in $R \in \Tight{+}$}\label{alg:adm-path-outtight}
\KwData{A set $R \in \R \cap \Tight{+}$}
\KwResult{A  safe source $s$ in $S\in \Minimal{-}$ with $S \subseteq R$, a safe sink $t$ in $T \in \Minimal{+}$ with $T\subseteq R$, and an $(s,t)$-hyperpath $P$ that does not leave $R$.}
Take a set $T \in \Minimal{+}$ with $T \subseteq R$\;
Take a safe sink $t$ in $T$\;
$Z := \{t\}$\;
$F := (Z, \emptyset)$\;
$V' := R$\;
\While{an $(X, v)$ hyperarc exists in $\vec{\Hg}$ such that $v\in Z$ and $(X \cap V') - Z\neq\emptyset$}{
    \While{$(X \cap V') - Z \neq \emptyset$}{
        Take $u$ in $(X \cap V') - Z$\;
        $Z := Z \cup \{u\}$\;
        $F := F + uv$\;
        \If{$Q^u_- \subsetneq V'$}{
            $V' := Q^u_-$\;
        }
    }
}
$S := V'$\;
Take a safe source $s$ in $S$\;
$P' := F[s, t]$\;
Let $P$ be the hyperpath in $\vec{\Hg}$ corresponding to $P'$\;
Return $S,T,s,t,P$\;
\end{algorithm}

\cref{alg:adm-path-outtight} works as follows. First, we choose a minimal set $T$ contained in $R$ (possibly $R$ itself), such that $T \in \Minimal{+}$ as well as a safe sink $t$ in $T$.
We keep track of a set $Z$ of vertices we already explored (that contains $t$ initially), a \emph{search $t$-in-arborescence} $F$ and a set $V'$ of vertices that will potentially be explored and which is initialized to $R$.
While there is a hyperarc $(X, v)$ entering $Z$ and whose tail vertex set $X$ contains unexplored vertices in $V'$, we consider all possible trimmings of the hyperarc $(X, v)$ into $(u, v)$, where $u \in (X \cap V') - Z$. We mark the tail vertex $u$ as explored, trim the hyperarc $(X, v)$ into the arc $uv$, and add it to $F$.
We ensure that if our search enters a set in $\Tight{-}$, we do not leave that set (see Line 11 of \cref{alg:adm-path-outtight}). We will see that the last set in $\Tight{-}$ entered by the search is in $\Minimal{-}$ and we let $S$ be that set.
We find a safe source $s$ in $S$.
Finally, we let $P'$ be the unique directed $(s, t)$-path $F[s, t]$. Notice that $P'$ is the trimming of an $(s, t)$-hyperpath $P$ and the vertex set of $P'$ is contained in $R$. The algorithm returns $S, T, s, t, P$ and terminates.
\medskip

The proof of~\cref{lem:algo_outtight} is similar to that of~\cref{lem:algo_intight}.

\begin{lemma}
\label{lem:algo_outtight}
    If $R \in \R \cap \Tight{+}$, then there exist  $S \in \Minimal{-}$, $T \in \Minimal{+}$, a safe source $s$ in $S$, a safe sink $t$ in $T$,  and an $(s,t)$-hyperpath $P$ such that  $S \subsetneq R$, $T \subseteq R$,  $P$ does not leave $R$,  $Q^s_+ = R$ and $\delta^-_{P'}(X) = \emptyset$ for every $X \in \Tight{-}$ with $s,t \not\in X$  where $P'$ is the unique directed $(s,t)$-path obtained from $P$ by trimming.
\end{lemma}

{
\begin{proof}
The algorithm starts by taking a safe sink in $T$. Such a safe sink exists by \cref{lem:safe_sink_exists}.

We show that \cref{alg:adm-path-outtight} provides the required hyperpath $P$.
Let $S' \in \Minimal{-}$ with $S' \subsetneq R$ and  $s'$ a safe sink in $S'$ by \cref{lem:safe_source_exists}. Then, by \cref{clm:inter_union_tight_is_tight}(a), $S' \subsetneq Q^{s'}_+\subseteq R.$ Thus, by minimality of $R$, $Q^{s'}_+=R.$ Then, by \cref{lem:path_in_qsp_qtm}(a), $t$ is reachable from $s'$ by a hyperpath $P$  that does not leave $Q^{s'}_+$. This implies that the algorithm will eventually enter a set $S \in \Minimal{-}$, containing a safe source $s$ by \cref{lem:safe_source_exists}. Thus, the algorithm terminates with a safe source $s$ in $S\in \mathcal{M}_-$ with $S \subsetneq R$, a safe sink $t$ in $T\in\mathcal{M}_+$ with $T \subseteq R$, and an $(s,t)$-hyperpath $P$ that does not leave $R$.

Recall that $P'=F[s,t].$
It remains to prove that for all $X \in \Tight{-}$ such that $s, t \not\in X$, $\delta^-_{P'}(X) = \emptyset$. Suppose that $\delta^-_{P'}(X) \neq \emptyset$. Since $t \not\in X$, the path $P'$ leaves $X$. Let $uv$ be the first arc of $P'$ that leaves $X$. Let $Y = Q^u_-$ and note that $P'$ leaves $Y$. As $u \in X \cap Y, X, Y \in \Tight{-}$, we get, by \cref{clm:inter_union_tight_is_tight}(a) on $X$ and $Y$  and the minimality of $Y$, that $Y \subseteq X \subsetneq V - t$. As $Y \in \Tight{-}$ and $s \in S \in \Minimal{-}$, we get, by \cref{clm:inter_union_tight_is_tight}(a), that $Y$ and $S$ are disjoint. If $Y - R \neq \emptyset$, then, by $Y \in \Tight{-}, R \in \Tight{+}, s \in R - Y$ and by \cref{clm:substraction_tight_is_tight}(c), we obtain that $R - Y \in \Tight{+}$. Since $S \subseteq R - Y$, this contradicts $R \in \R$. Hence,  $Q^u_-=Y \subsetneq R$. Then, since $Q^u_-\subseteq Q^v_-$, by the \emph{if} condition on line 12 of the \cref{alg:adm-path-outtight}, we have that $F$ and hence $P'$ contains no arc entering $Q^u_-$ and thus entering $X$. This contradicts $\delta^-_{P'}(X)  \neq \emptyset$, completing the proof.
\end{proof}
}

\section{Proving the admissibility of the hyperpath}
\label{sec:admissability}

In this section, we prove that Algorithms \ref{alg:adm-path-intight} and \ref{alg:adm-path-outtight} provide an admissible hyperpath.
\medskip

Let $R \in \R$ and $\Minimal{} \neq \{V\}$. By \cref{lem:algo_intight,lem:algo_outtight}, there exist $S$ in $\Minimal{-}$, $T$ in $\Minimal{+}$, with $S, T \subseteq R$, a safe source $s$ in $S$, a safe sink $t$ in $T$, and an $(s,t)$-hyperpath $P$ that does not leave $R$. Let $P'$ be the trimming of $P$ found by either \cref{alg:adm-path-intight} or \cref{alg:adm-path-outtight}. Furthermore, if $R \in \Tight{-}$, $\delta^+_{P'}(X) = \emptyset$ for all $X \in \Tight{+}$ with $s,t \not\in X$. Similarly, if $R \in \Tight{+}$, $\delta^-_{P'}(X) = \emptyset$ for all $X \in \Tight{-}$ with $s,t \not\in X$. It is clear, from the search technique of the algorithm, that $\ell := |\mathcal{A}(P)| < |V|$. For $i=0,1,\ldots,\ell,$ let $\vec{\Hg}^i$ be obtained from $\vec{\Hg}$ by reorienting the last (resp. first) $i$ hyperarcs $(A_j,a_j)$ of $P$ towards $a_{j-1}$ if $R \in \Tight{-}$ (resp. $R \in \Tight{+}$).
\medskip

We first prove that the hyperarc-connectivity never decreases when we reorient the hyperarcs of $P$.

\begin{lemma}
    For $i = 1, \ldots, \ell,~\lambda(\vec{\Hg}^i) \geq k$.
\end{lemma} 

\begin{proof}
    Suppose that there exists $i$ and $X \subseteq V - r$ such that either (a) $d^+_{\vec{\Hg}^i}(X) < k$ or (b) $d^-_{\vec{\Hg}^i}(X) > k$. We take the first of such $i$.

\begin{itemize}
       \item[(a)] 
    		\begin{itemize}
        			\item If $R \in \Tight{-}$ then, by $\lambda({\vec{\Hg}^{i-1}}) \geq k$, we get that $t \not\in X \in \Tight{+}$ and $d^+_{P'}(X) \neq \emptyset$. By the definition of $P$ and \cref{lem:algo_intight}, we have $s \in X$. Thus, by \cref{lem:safe_to_safe_kp1}, since we are reorienting hyperarcs along a path, we have a contradiction: $k + 1  \leq d^+_{\vec{\Hg}^{0}}(X) \leq d^+_{\vec{\Hg}^i}(X) + 1 \leq k.$
       			\item If $R \in \Tight{+}$ then, by $\lambda({\vec{\Hg}^{i-1}}) \geq k$, we get $s \in X \in \Tight{+}$. By $s \in R \in \Tight{+}$ and \cref{clm:inter_union_tight_is_tight}(b), we get that $X \cap R \in \Tight{+}$. By \cref{lem:algo_outtight}, we have $R = Q^s_+$ and $\delta^+_{P'}(X) \neq \emptyset$. Hence we have $R = Q^s_+ \subseteq X \cap R \subsetneq R$, a contradiction.
    		\end{itemize}
       \item[(b)] 
      		\begin{itemize}
        			\item If $R \in \Tight{-}$ then, by $\lambda({\vec{\Hg}^{i-1}}) \geq k$, we get $t \in X \in \Tight{-}$. By $t \in R \in \Tight{-}$ and \cref{clm:inter_union_tight_is_tight}(a), we get that $X \cap R \in \Tight{-}$. By \cref{lem:algo_intight}, we have  $R = Q^t_-$ and $\delta^-_{P'}(X) \neq \emptyset$. Hence we have $R = Q^t_- \subseteq X \cap R \subsetneq R$, a contradiction.
        			\item If $R \in \Tight{+}$, by $\lambda({\vec{\Hg}^{i-1}}) \geq k$, we get that $s \not\in X \in \Tight{-}$ and $d^-_{P'}(X) \neq \emptyset$. By the definition of $P$ and \cref{lem:algo_outtight}, we have $t \in X$. Thus, by \cref{lem:safe_to_safe_kp1}, since we are reorienting hyperarcs along a path, we have a contradiction: $k + 1\leq d^-_{\vec{\Hg}^{0}}(X) \leq d^-_{\vec{\Hg}^i}(X) + 1 \leq k.$
    		\end{itemize}
\end{itemize}
\end{proof}

Let $\vec{\Hg}' = \vec{\Hg}^{\ell}$. We define the sets $\Tight{-}', \Tight{+}'$ and $\Minimal{}'$ for $\vec{\Hg}'$ as the sets $\Tight{-}, \Tight{+}$ and $\Minimal{}$ were defined for $\vec{\Hg}$. Remark that if $R$ is in $\Tight{-}$ (resp. $\Tight{+}$) then by \cref{lem:algo_intight} (resp. \cref{lem:algo_outtight}) we get that $R$ is in $\Tight{-}'$ (resp. $\Tight{+}')$.
For $P$ to be admissible, it remains to show the following.

\begin{lemma}
\label{lem:value_decreasing}
    Either (a) $|\Minimal{}'| < |\Minimal{}|$ or (b) $|\Minimal{}'| = |\Minimal{}|$ and $\sum_{X \in \Minimal{}'} |X| > \sum_{X \in \Minimal{}} |X|$ holds.
\end{lemma}

\begin{proof}
If $T \neq R \neq S$ then, by \cref{clm:substraction_tight_is_tight}(c) and the minimality of $S$ and $T$, we get that $S$ and $T$ are disjoint. Hence, since $s \in S$ and $t \in T$, the path $P'$ leaves $S$ and enters $T$. If $S = R$ then, since $s$ is a safe source in $R$, $T \in \Tight{+}$ and $T \subsetneq R$, we have $s \not\in T$. Hence, since $t \in T$, the path $P'$ enters $T$. If $T = R$, then, since $t$ is a safe sink in $R$, $S \in \Tight{-}$ and $S \subsetneq R$, we have $t \not\in S$. Hence, since $s \in S$, the path $P'$ leaves $S$.
\noindent
Since $\Hg$ is $(k+1, k+1)$-partition-connected, we have
$d^-_{\vec{\Hg}'}(T) = d_\Hg(T) - d^+_{\vec{\Hg}'}(T) \geq (2k+2) - (k + 1) = k+1$ if $P'$ enters $T$ and 
$d^+_{\vec{\Hg}'}(S) = d_\Hg(S) - d^-_{\vec{\Hg}'}(S) \geq (2k+2) - (k + 1) = k+1$ if $P'$ leaves $S$.
Thus, $S \in \Minimal{} - (\Tight{-}' \cup \Tight{+}')$ if $P'$ leaves $S$, and $T \in \Minimal{} - (\Tight{-}' \cup \Tight{+}')$ if $P'$ enters $T$.
We now characterize every element of $\Minimal{}' - \Minimal{}$ in the following two claims.

\begin{claim}
\label{lem:new_min_is_R}
    If $X \in (\Tight{-} \cup \Tight{+}) \cap (\Minimal{}' - \Minimal{})$, then $X = R$.
\end{claim}

\begin{claimproof}

Since $X \in (\Tight{-} \cup \Tight{+}) \cap (\Tight{-}' \cup \Tight{+}')$, either $s,t \not\in X$ or $s,t \in X$. In the former case, $X \in \Minimal{}'$ implies that $X \in \Minimal{}$ which is a contradiction. Hence, we consider $s,t \in X$. We distinguish some cases.

\begin{itemize}
    \item[(a)]  If $X \in \Tight{-}$ then since $t$ is a safe sink in $T$, we get $T \subsetneq X$.
\begin{itemize}
    \item[-] If $R \in \Tight{-}$ then we know, by \cref{clm:substraction_tight_is_tight}(a), that $s,t \in X \cap R \in \Tight{-}$, so we have $X \cap R \in \Tight{-}'$. Thus, by $X \in \Minimal{}'$, we obtain that $X = X \cap R \subseteq R$. Hence, by $T \in \Tight{+}$, $X \in \Tight{-}$ and $T \subsetneq X \subseteq R \in \R$, we have $X = R$.
    \item[-] If $R \in \Tight{+}$ then if $R - X \neq \emptyset \neq X - R$, by \cref{clm:substraction_tight_is_tight}(c), we get  $X - R \in \Tight{-}'$ which contradicts $X \in \Minimal{}'$. Hence, either $R \subseteq X$ or $X \subseteq R$. If $R = X$, we have a contradiction from the $(k+1, k+1)$-partition-connectivity of $\Hg$. Hence, in the former case, we get $X \not\in \Minimal{}'$ and in the latter case, by $T \in \Tight{+}$, $X \in \Tight{-}$ and $T \subsetneq X \subsetneq R$, we get $R \not\in \R$. In both cases, we have a contradiction.
\end{itemize}
    \item[(b)]  If $X \in \Tight{+}$ then since $s$ is a safe source in $S$, we get $S \subsetneq X$.
\begin{itemize}
    \item[-] If $R \in \Tight{-}$ then if $R - X \neq \emptyset \neq X - R$, by \cref{clm:substraction_tight_is_tight}(c), we get $X - R \in \Tight{+}'$ which contradicts $X \in \Minimal{}'$. Hence, either $R \subseteq X$ or $X \subseteq R$. If $R = X$, we have a contradiction from the $(k+1, k+1)$-partition-connectivity of $\Hg$. Hence, in the former case, we get $X \not\in \Minimal{}'$ and in the latter case, by $S \in \Tight{-}$, $X \in \Tight{+}$ and $S \subsetneq X \subsetneq R$, we get $R \not\in \R$. In both cases, we have a contradiction.
    \item[-] If $R \in \Tight{+}$ then  we know, by \cref{clm:substraction_tight_is_tight}(b), that  $s,t \in X \cap R \in \Tight{+}$, so we have $X \cap R \in \Tight{+}'$. Thus, by $X \in \Minimal{}'$, we obtain that $X = X \cap R \subseteq R$. Hence, by $S \in \Tight{-}$, $X \in \Tight{+}$ and $S \subsetneq X \subseteq R \in \R$, we have $X = R$.
\end{itemize}
\end{itemize}
\end{claimproof}

\begin{claim}
\label{lem:new_tight_is_contained}
    If $X \in \Minimal{}' - (\Tight{-} \cup \Tight{+})$, then $S \subsetneq X \subsetneq R$ or $T \subsetneq X \subsetneq R$.
\end{claim}

\begin{claimproof}
Since $X \in \Minimal{}' - (\Tight{-} \cup \Tight{+})$, either (a) $s \not\in X, t \in X \in \Dangerous{-}$ or (b) $t \not\in X, s \in X \in \Dangerous{+}$ holds.

\begin{itemize}
    \item[(a)] We first show that $T \subsetneq X$. Indeed, if $T - X \neq \emptyset$, then by the fact that $t$ is a safe sink, there exists $Y \subseteq X - t$ with $Y \in \Tight{-}$. Since $s,t \not\in Y$, $Y \in \Tight{-}'$ contradicting $X \in \Minimal{}'$.
	\begin{itemize}
    		\item If $R\in\Tight{-}$ then $R\in\Tight{-}'$, so by Claim 1(a) and $X\in\Minimal{-}',$ we get that $X\subseteq R.$  Since $s \in R - X$, we obtain that $X \subsetneq R$.
    		\item If $R \in \Tight{+}$ then, by $X \in \Minimal{}' \cap \Tight{-}'$, $R \in \Tight{+}'$ and \cref{clm:substraction_tight_is_tight}(c) on $X$ and $R$, we get that $X$ and $R$ are not crossing. Thus, by $s\in R-X, t\in X\cap R$ and $r\in V-(X\cup R)$,  we have $X \subsetneq R$.
	\end{itemize}
    \item[(b)] We first show that  $S \subsetneq X$.  Indeed, if $S - X \neq \emptyset$, then by the fact that $s$ is a safe source, there exists $Y \subseteq X - s$ with $Y \in \Tight{+}$. Since $s,t \not\in Y$, $Y \in \Tight{+}'$ contradicting $X \in \Minimal{}'$.
	\begin{itemize}
   		\item  If $R \in \Tight{-}$ then, by $X \in \Minimal{}' \cap \Tight{+}'$, $R \in \Tight{-}'$ and \cref{clm:substraction_tight_is_tight}(c) on $X$ and $R$, we get that $X$ and $R$ are not crossing. Thus, by $s\in R-X, t\in X\cap R$ and $r\in V-(X\cup R)$,  we have $X \subsetneq R$. 
    		\item If $R \in \Tight{+}$ then $R\in\Tight{+}'$, so by \cref{clm:inter_union_tight_is_tight}(b) and $X\in\Minimal{+}',$ we get that $X \subseteq R$. Since $t \in R - X$, we obtain that $X \subsetneq R$.
	\end{itemize}
\end{itemize}
\end{claimproof}

By \cref{lem:new_min_is_R,lem:new_tight_is_contained}, every element of $\Minimal{}' - \Minimal{}$ contains at least an element of $\Minimal{} - \Minimal{}'$ strictly ($S$ or $T$). Furthermore, by \cref{clm:inter_union_tight_is_tight,clm:substraction_tight_is_tight}, the elements of $\Minimal{}' - \Minimal{}$ are disjoint. Thus, $|\Minimal{}' - \Minimal{}| \leq |\Minimal{} - \Minimal{}'|$, and if $|\Minimal{}' - \Minimal{}| = |\Minimal{} - \Minimal{}'|$, we have $\sum_{X \in \Minimal{}'} |X| > \sum_{X \in \Minimal{}} |X|$.
\end{proof}

\section{Minimum $(s,t)$-separators in directed hypergraphs}
\label{sec:separators}

For directed graphs, it is well-known that the max flow-min cut algorithm of Edmonds-Karp \cite{edmonds1972theoretical} provides 
a minimum out-degree $(s,t)$-separator that is contained in all minimum out-degree $(s,t)$-separators.
By submodularity, for directed hypergraphs $\vec{\Hg}$ also holds that for every $s,t\in V,$ there exists a minimum out-degree $(s,t)$-separator that is contained in all minimum out-degree $(s,t)$-separators. We denote that unique set by $S_{s,t}^{\vec{\Hg}}.$ 
Similarly, there exists a minimum in-degree $(t,s)$-separator that is contained in all minimum in-degree $(t,s)$-separators. We denote that unique set by $T_{t,s}^{\vec{\Hg}}.$
The question is how to find them in polynomial time.
For this purpose we provide a simple reduction to the same problem on directed graphs.
In a similar way, a reduction for general directed capacitated hypergraphs can be obtained.
For $\vec{\Hg}=(V,\mathcal{A}),$ we consider a directed  graph that is obtained from its directed bipartite incidence graph $\vec{G} = (V', A')$ as follows. We let $V' = V \cup \{ w_{(X, v)} : (X, v) \in \mathcal{A} \}$. The arc set $A'$ contains for every $(X, v) \in \mathcal{A},$ the arc $w_{(X, v)}v$ and $|\mathcal{A}|+1$ parallel arcs $xw_{(X, v)}$ for every $x \in X$. 
Note that $|V'|=|V|+|\mathcal{A}|$ and $|A'|\le |V||\mathcal{A}|(|\mathcal{A}|+1).$ An example of this construction is shown in \cref{fig:reduction_hyperarc_to_arc}. Let $\cev{G}$ be the directed graph obtained from $\vec{G}$ by changing the orientation of all arcs.

\begin{figure}[ht]
    \centering
\begin{tikzpicture}[scale=0.75]
	\begin{pgfonlayer}{nodelayer}
		\node [style=wideBase] (0) at (-5.5, 2) {$u_1$};
		\node [style=wideBase] (1) at (-5.5, 0) {$u_2$};
		\node [style=wideBase] (2) at (-5.5, -2) {$u_3$};
		\node [style=wideBase] (3) at (-1, 0) {$v$};
		\node [style=none] (4) at (-3, 0) {};
		\node [style=none] (5) at (-2.75, 0.5) {$(X,v)$};
		\node [style=wideBase] (6) at (1, 2) {$u_1$};
		\node [style=wideBase] (7) at (1, 0) {$u_2$};
		\node [style=wideBase] (8) at (1, -2) {$u_3$};
		\node [style=wideBase] (9) at (6, 0) {$v$};
		\node [style=wideBase,inner sep=1pt] (10) at (3.5, 0) {$w_{(X,v)}$};
		\node [style=none] (11) at (0, 0) {$\implies$};

        \node [style=none] (12) at (3.25, 1.75) {\footnotesize $|\mathcal{A}|+1$};
        \node [style=none] (13) at (1.95, 0.5) {\footnotesize $|\mathcal{A}|+1$};
        \node [style=none] (14) at (3.25, -1.75) {\footnotesize $|\mathcal{A}|+1$};
	\end{pgfonlayer}
	\begin{pgfonlayer}{edgelayer}
		\draw [style=wideLine, in=180, out=0] (2) to (4.center);
		\draw [style=wideLine] (1) to (4.center);
		\draw [style=wideLine, in=-180, out=0] (0) to (4.center);
	 	\draw [style=arrow] (4.center) to (3);
		\draw [style=arrow, in=-120, out=0] (8) to (10);
		\draw [style=arrow] (7) to (10);
		\draw [style=arrow, in=120, out=0] (6) to (10);
		\draw [style=arrow] (10) to (9);
	\end{pgfonlayer}
\end{tikzpicture}
    \caption{Arcs in $\vec{G}$ incident to $w_{(X,v)}$ resulting from a hyperedge $(X,v)$ of $\vec{\Hg}$}
    \label{fig:reduction_hyperarc_to_arc}
\end{figure}
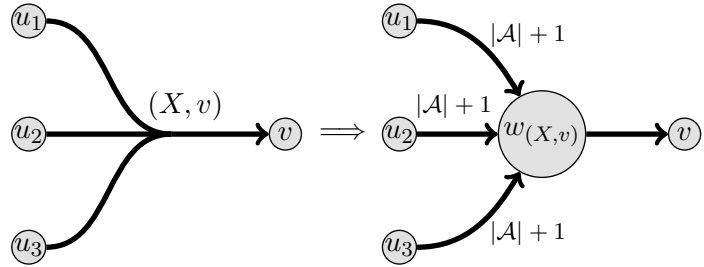

\begin{lemma}\label{separator1}
$S_{s,t}^{\vec{\Hg}}=V\cap S_{s,t}^{\vec G}$
 and $T_{t,s}^{\vec{\Hg}}=V\cap S_{t,s}^{\footnotesize{\cev{G}}}.$
\end{lemma}

\begin{proof}
Let $Z'=S_{s,t}^{\vec G}$, $Z = V \cap Z'$ and $X=S_{s,t}^{\vec{\Hg}}$.  Since $Z'$ is an $(s, t)$-separator in $\vec{G}$, we have  $s \in Z\subseteq V-t$.
For a set $W\subseteq V,$ let $\varphi(W)=W\cup\{w_{(Y,v)}:(Y,v)\in\mathcal{A}, W\cap Y\neq\emptyset\}.$ 
Note that, by the definition of $Z,$ we have  $\varphi(Z)\subseteq Z'.$
Let $X'=\varphi(X)$.
Note that if an arc leaves $X'$ in $\vec{G}$ then there exists a directed hyperedge that leaves $X$ in $\vec{\Hg}$. Note also that if a directed hyperedge leaves $Z$ in $\vec{\Hg}$  then there exists an arc that leaves $Z'$ in $\vec{G}$. Then, by the definitions of $Z'$ and $X$, we have $d^+_{\vec{G}}(Z')\le d^+_{\vec{G}}(X')\le d^+_{\vec{\Hg}}(X)\le d^+_{\vec{\Hg}}(Z)\le d^+_{\vec{G}}(Z').$ It follows that equality holds everywhere, in particular we get that $X'$ is a minimum out-degree $(s,t)$-separator in $\vec{G}$ and $Z$ is a minimum out-degree $(s,t)$-separator in $\vec{\Hg}$.
Then we have $Z'\subseteq X'$ and $X\subseteq Z$. Thus, by $\varphi(Z)\subseteq Z'$, we get  $X'=\varphi(X)\subseteq\varphi(Z)\subseteq Z'\subseteq X'.$ It follows that equality holds everywhere, in particular we get that $\varphi(X)=\varphi(Z)$ and hence $X=Z.$

\medskip

Let $U'=S_{t,s}^{\footnotesize{\cev{G}}}$, $U := V\cap U'$ and $X=T_{t,s}^{\vec{\Hg}}$.    Since $U'$ is a $(t, s)$-separator in $\cev{G}$, we have that $t \in U\subseteq V-s$.     For a set $W\subseteq V,$ let $\phi(W)=W\cup\{w_{(Y,v)}:(Y,v)\in\mathcal{A}, Y\cup\{v\}\subseteq W\}.$ Note that, by the definition of $U,$ we have  $\phi(U)\subseteq U'.$
Let  $X'=\phi(X)$.
Note that if an arc leaves $X'$ in $\cev{G}$ then there exists a directed hyperedge that enters $X$ in $\vec{\Hg}$. Note also that if a directed hyperedge enters $U$ in $\vec{\Hg}$  then there exists an arc that leaves $U'$ in $\cev{G}$.
Then, by the definitions of $U'$ and $X$, we have $d^+_{{\footnotesize\cev{G}}}(U')\le d^+_{{\footnotesize\cev{G}}}(X')\le d^-_{\vec{\Hg}}(X)\le d^-_{\vec{\Hg}}(U)\le d^+_{\vec{G}}(U').$
It follows that equality holds everywhere, in particular we get that $X'$ is a minimum out-degree $(t, s)$-separator in $\cev{G}$ and $U$ is a minimum in-degree $(t, s)$-separator in $\vec{\Hg}$.
Then we have $U'\subseteq X'$ and $X\subseteq U$.
Thus, by $\phi(U)\subseteq U'$, we get  $X'=\phi(X)\subseteq\phi(U)\subseteq U'\subseteq X'.$ It follows that equality holds everywhere, in particular we get that $\phi(X)=\phi(U)$ and hence $X=U.$
\end{proof}

We can hence apply the algorithm of Edmonds-Karp \cite{edmonds1972theoretical} to compute in polynomial time the sets $S_{s,t}^{\vec G}$ and $S_{t,s}^{\footnotesize{\cev{G}}}$ and obtain, by  \cref{separator1}, the sets $S_{s,t}^{\vec{\Hg}}.$ and $T_{t,s}^{\vec{\Hg}}.$

\section{Complexity of the algorithms}
\label{sec:complexity}

In order to compute the sets $\R$, $\Minimal{-}$, and $\Minimal{+}$ used in \cref{alg:conn-augmentation-global},  we apply the  algorithms of \cref{sec:separators}. 

\begin{lemma}
\label{lem:compute_sets_polytime}
    There is an algorithm that computes the sets $\R$, $\Minimal{-}$, $\Minimal{+}$ in time polynomial in the size of $\Hg$.
\end{lemma}

\begin{proof}
    Let $\mathcal{Q} = \bigcup_{v\in V} \{Q^v_-\} \cup \{Q^v_+\}$. We may obtain $\mathcal{Q}$ in polynomial time by applying the algorithms of \cref{sec:separators} to compute $S_{u,v}^{\vec{\Hg}}$ and $T_{u,v}^{\vec{\Hg}}$ for every pair $(u, v)$ of vertices of $\vec{\Hg}$.
    
    We show that $\R, \Minimal{-}, \Minimal{+} \subseteq \mathcal{Q}$.
    If $v \in S \in \Minimal{-}$ then $Q^v_- = S$ and if $v \in T \in \Minimal{+}$ then $Q^v_+ = T$, so $\Minimal{-}, \Minimal{+} \in \mathcal{Q}$.
    Let $R \in \R$. If $R \in \Tight{-}$ then $R$ contains $T \in \Minimal{+}$ and by \cref{lem:safe_sink_exists} there exists a safe sink $t$ in $T$. We have seen in the proof of \cref{alg:adm-path-intight} that $Q^t_- = R$. Similarly, if $R \in \Tight{+}$ then $R$ contains $S \in \Minimal{-}$ and by \cref{lem:safe_source_exists} there exists a safe source $s$ in $S$. We have seen in the proof of \cref{alg:adm-path-outtight} that $Q^t_+ = R$. Therefore, $\R \subseteq \mathcal{Q}$.

    Finally, we can iteratively eliminate the sets from $\mathcal{Q}$ not in $\R, \Minimal{-}$ and $\Minimal{+}$ in polynomial time.
\end{proof}

\begin{lemma} \label{lem:polynomial}
    \cref{alg:adm-path-intight,alg:adm-path-outtight} run in time polynomial in the size of the input $\Hg$.
\end{lemma}

\begin{proof}
Both algorithms execute a simple search which runs in polynomial time in the number of hyperedges.
It remains to show we can find a safe source and a safe sink in polynomial time.
By symmetry, we only show how to find a safe source in a given $S \in \Minimal{-}$.
For $u,v \in S$, let $X^u_v=S_{u,v}^{\vec{\Hg}}$.

\begin{claim}
The vertex $u$ is not a safe source in $S$ if and only if there exists a vertex $v \in S-u$ such that either $X^u_v\in\Tight{+}$ or $X^u_v \in \Dangerous{+}$ and $X^u_v$ does not contain a set $Z \in \Tight{+}$. 
\end{claim}

\begin{claimproof}
The sufficiency follows from the definition of a safe source. To prove the necessity, we suppose that $u$ is not a safe source in $S$ and $X^u_v \notin \Tight{+}$ for all $v\in S-u$. Then there exists by definition, for some vertex $v \in S-u$, a minimum out-degree $(u,v)$-separator $Y^u_v$ such that $Y^u_v \in \Dangerous{+}$ where there is no set $Z \in \Tight{+}$ such that $Z \subseteq Y^u_v - u$. 
Since $X_u^v\subseteq Y_u^v$, we get that $X^u_v$ does not contain a set $Z \in \Tight{+}$.
\end{claimproof}

To compute whether $u$ is a safe source in $S$ or not, it then suffices, for all $v \in S-u$, to check whether $X^u_v \in \Tight{+}$ or whether $X^u_v \in \Dangerous{+}$ contains a set $Z \in \Tight{+}$, which can be done in polynomial time. By \cref{lem:safe_source_exists}, $S$ contains a safe source, thus there exists at least one vertex $s \in S$ that does not satisfy the conditions above which can be checked in polynomial time.
\end{proof}

It follows from the previous lemmas that \cref{alg:conn-augmentation-global} runs in polynomial time.

\section{Concluding remarks}

We generalized the results of Ito et al.~\cite{ito2022monotone} to hypergraphs and gave an algorithmic proof of \cref{KPP_hypergraph}. Several interesting questions remain open. We do not know if the $|V|^3$ bound on the number of hyperarc reorientations is tight. Generalizing this result to general directed hypergraphs with tail vertex sets and head vertex sets for each hyperarc would also provide a general characterization of $k$-hyperarc-connected orientable hypergraphs, which is not known to our knowledge.

\bibliographystyle{plainurl}
\bibliography{biblio}

\end{document}